\documentclass[11pt]{amsart}

\usepackage[margin=1.2in]{geometry}
\title{The genus of curve, pants and flip graphs}
\author{Hugo Parlier and Bram Petri}
\date{\today}
\thanks{Research supported by Swiss National Science Foundation grant number PP00P2\_128557.}

\newtheorem{thm}{Theorem}[section]
\newtheorem{prp}[thm]{Proposition}
\newtheorem{cor}[thm]{Corollary}
\newtheorem{lem}[thm]{Lemma}

\newtheorem*{thmA}{Theorem A}
\newtheorem*{thmB}{Theorem B}
\newtheorem*{thmC}{Theorem C}

\newtheorem{dff}{Definition}[section]

\usepackage{amsmath,amsfonts,amssymb,float,graphicx,enumitem}

\newcommand{\Pro}[2]{\mathbb{P}_{#1}\left[#2\right]}
\newcommand{\ProC}[3]{\mathbb{P}_{#1}\left[#2\mid \;#3\right]}
\newcommand{\ExV}[2]{\mathbb{E}_{#1}\left[#2\right]}
\newcommand{\verz}[2]{\left\{ #1 \,:\, #2\right\}}
\newcommand{\rij}[3]{\left\{#1\right\}_{#2}^{#3}}

\newcommand{\abs}[1]{\left| #1 \right|}
\newcommand{\aant}[1]{\#\left( #1 \right)}

\newcommand{\floor}[1]{\left\lfloor #1 \right\rfloor}
\newcommand{\ceil}[1]{\left\lceil #1 \right\rceil }

\newcommand{\symm}[1]{\mathrm{S}_{#1}}

\newcommand{\CG}[1]{\mathcal{C}\left(#1\right)}
\newcommand{\MC}[1]{\mathcal{MC}\left(#1\right)}

\newcommand{\PG}[1]{\mathcal{P}\left(#1\right)}
\newcommand{\MP}[1]{\mathcal{MP}\left(#1\right)}

\newcommand{\FG}[1]{\mathcal{F}\left(#1\right)}
\newcommand{\MF}[1]{\mathcal{MF}\left(#1\right)}

\newcommand{\MCG}[1]{\mathrm{Mod}\left(#1\right)}

\begin{document}

\maketitle


\begin{abstract}
This article is about the graph genus of certain well studied graphs in surface theory: the curve, pants and flip graphs. We study both the genus of these graphs and the genus of their quotients by the mapping class group. The full graphs, except for in some low complexity cases, all have infinite genus. The curve graph once quotiented by the mapping class group has the genus of a complete graph so its genus is well known by a theorem of Ringel and Youngs. For the other two graphs we are able to identify the precise growth rate of the graph genus in terms of the genus of the underlying surface. The lower bounds are shown using probabilistic methods. 
\end{abstract}

\section{Introduction}

A certain number of simplicial complexes and graphs have been very useful in the study of Teichm\"uller spaces, mapping class groups and hyperbolic $3$-manifolds. These include the curve, pants, and arc complexes and the related curve, pants and flip graphs. 
Viewed as metric spaces, the large scale geometry of these spaces has been extensively investigated. They all have (with the exception of some low complexity examples) the mapping class group as set of simplicial automorphisms but their geometries are very different. We'll be interested in a topological property: the {\it graph genus} of the three graphs as well as the genus of their quotients by the mapping class group.

The curve graph records the isotopy classes of curves on a fixed surface and their intersection properties. Curve graphs have played a key role in the understanding of the geometry of Teichm\"uller space, hyperbolic $3$-manifolds and mapping class groups (see for example \cite{Min1}, \cite{Min2}, \cite{Raf}, \cite{Ham}, \cite{BMM}, \cite{BCM} and \cite{MM2}). The geometry of the curve graph in particular has been studied in detail - it is a Gromov hyperbolic metric space and quasi-isometric to an electrified version of Teichm\"uller space \cite{MM2}. More recently it was proved independently by Aougab in \cite{Aou}, Bowditch in \cite{Bow2}, and Clay, Rafi and Schleimer \cite{CRS} that the hyperbolicity constant is in fact uniform (independent of the topology). In \cite{HPW}, Hensel, Webb and Przytycki prove that the constant is at most $17$. 

The mapping class group of a surface acts naturally on the corresponding curve graph through its action on isotopy classes of curves. The quotient graph under the action of the mapping class group is what we will call the modular curve graph. It essentially follows from the classification of finite type surfaces that, at least for closed surfaces, this is a complete graph with additional loops attached to vertices. With this observation, the graph genus of the modular curve graph is equivalent to a well-known theorem of Ringel and Youngs \cite{RY} about the genus of the complete graph. In some sense this is one of our starting points; for the other two graphs we consider the question does not boil down to a well known theorem.

Pants graphs have pants decompositions as vertices and have edges between vertices if the two corresponding pants decompositions can be related by a so-called elementary move. Although related to the curve graph, the pants graph has a very different geometry and is only Gromov hyperbolic in a few low complexity cases, as proved by Brock and Farb in \cite{BF}. Brock also proved that it is quasi-isometric to Teichm\"uller space with the Weil-Petersson metric. The mapping class group acts on the pants graph as well and the quotient of this action, the modular pants graph, has a purely graph theoretic interpretation. Vertices are homeomorphism types of pants decompositions and so can be encoded by isomorphism types of trivalent graphs. Elementary moves on pants decompositions then correspond to natural transformation of these graphs (sometimes called $L$-moves or Whitehead moves). 

The final type of graphs that we consider is formed by the flip graphs. These are graphs of triangulations of punctured surfaces, which share an edge if they are related by a so called flip (or equivalently if they differ by a single arc). These graphs are related to a number of different topics including cell decompositions of {\it decorated} Teichm\"uller space (see for example \cite{Pen1} \cite{Pen2}). Again, the mapping class group is generally the full automorphism group of the graph, a result of Korkmaz and Papadopoulos \cite{KP} recently extended to injective maps between different flip graphs by Aramayona, the first author and Koberda in \cite{AKP}. In the particular case where the surface is a polygon with $n$ sides, the geometry of these flip graphs was investigated by Sleator, Tarjan and Thurston in \cite{STT}. They proved that the diameter of these graphs was $2n-10$ for large enough $n$ (this was recently found to be true for $n\geq 12$ by Pournin \cite{Pou}). In general these graphs aren't finite though. However, the quotient of this graph by the mapping class group - the modular flip graph - is finite and every vertex in the flip graph has finite degree. The \v{S}varc-Milnor lemma then implies that the flip graph is a coarse model for studying the geometry of the mapping class group. 

When describing these three graphs, a first example is often the torus with a marked point or puncture. The pants graph and the curve graph coincide in this case and are isomorphic to a Farey graph. Via the relationship between curves and arcs on a once punctured torus, it is not too difficult to see that the corresponding flip graph is the dual graph to the Farey graph (with its usual disk embedding). This dual graph is an infinite trivalent tree. So in all three cases, the graphs are planar. Our first observation is that in general these graphs are far from being planar.
\begin{thmA}
Except for in a limited number of exceptional cases, the curve graph, the pants graph and the flip graph all have infinite genus.
\end{thmA}
We treat the graphs separately so the above theorem is a combination of Theorems \ref{thm_cg}, \ref{thm_pg} and \ref{thm_fg}.

Once we quotient the graphs by their automorphisms, the resulting graphs are finite so they have finite genus. Let $\Sigma_g$ be a closed genus $g$ surface and $\Sigma_{g,1}$ be a surface of genus $g$ with a single marked point. As mentioned above, the theorem of Ringel and Youngs about the genus of complete graphs and some simple observations about topological types of curves imply that  implies that the genus of the modular curve graph is exactly
$$
\gamma(\MC{\Sigma_{g}}) = \ceil{\frac{(\floor{\frac{g}{2}}-2)(\floor{\frac{g}{2}}-3)}{12}}.
$$
We're able to identify how the genus grows in function of the topology of the underlying surface for the other two graphs. 
\begin{thmB} The genus of the modular pants graph of $\Sigma_g$ grows like
$$
 \frac{1}{\sqrt{2g-2}}\cdot\left(\frac{6g-6}{4e} \right)^g
$$
for $g\rightarrow\infty$.

The genus of the modular flip graph of $\Sigma_{g,1}$ grows like 
$$
 \frac{1}{(4g-2)^{3/2}} \left(\frac{12g-6}{e}\right)^{2g}.
$$
for $g\rightarrow\infty$.\end{thmB}
In fact, what we prove is precise estimates on the growth rates with explicit constants.  The  more precise statements are those of Theorems \ref{thm_mp} and \ref{thm_mf}. The upper bounds are obtained using standard techniques to bound the genus of a graph. The lower bounds are much more difficult and we use probabilistic arguments in both cases. To do so, we need to adapt results about the probability that a random graph has an automorphism. As a by-product of what we need to prove, we end up counting the number of isomorphism classes of cubic multigraphs (by multigraph we mean a graph where one is allowed multiple edges between vertices and loops). Since the first version of the paper was released, we found a result of Wormald that can be used to do this count. However, our proof is different so we think it might be of independent interest.
\begin{thmC} \cite{Wor2} Let $I_N$ denote the number of isomorphism classes of cubic multigraphs on $N$ vertices. Then:
$$
I_N \sim \frac{e^{2}}{\sqrt{\pi N}} \left(\frac{3N}{4e}\right)^{N/2}
$$
for $N\rightarrow\infty$.
\end{thmC}
Wormald actually stated this theorem for general $k$-regular multigraphs and also the techniques in our proof generalize to higher bounded degree, but because we need the result only in the case of degree $3$, we restrict to this case.

The article is organized as follows. We begin with a preliminary section where we recall definitions and basic results about the genus of a graph and the three graphs we'll be interested in. We then have a section about random graphs where we recall some of the tools we need and where we prove others including Theorem C above. In Section 4 we discuss random triangulations; in particular we explain and adapt previous results of the second author and other known results. The subsequent sections are dedicated, in order, to our results concerning the curve graph, the pants graph and the flip graph. 

\section{Preliminaries}

In this section we describe the three types of graphs we consider and recall the definition of the genus of a graph. We begin with the graph genus.

\subsection{The genus of a graph} 

The genus of a graph is a measure of its topological complexity and is defined as follows:
\begin{dff} Let $\Gamma$ be a graph. The genus of $\Gamma$ is given by:
$$
\gamma(\Gamma) = \min\verz{g}{\text{there exists a continuous injection }f:\Gamma\rightarrow\Sigma_g}
$$
where the minimum of the empty set is $\infty$.
\end{dff}

The main tool we will use to get bounds on the genera of graphs is the following proposition. The lower bound is due to Beineke and Harary and  the upper bound is a well known fact that essentially bounds the genus of a graph in terms of its Betti number (or cycle rank):
\begin{prp}\cite{BH1}\label{prp_general} Let $\Gamma$ be a connected graph with $p$ vertices, $q$ edges and girth $h$ then:
$$
1+\frac{1}{2}\left(1-\frac{2}{h}\right)q-\frac{1}{2}p \; \leq \; \gamma(\Gamma) \; \leq \; \frac{1}{2} + \frac{1}{2}q-\frac{1}{2}p
$$
\end{prp}

The proof of the lower bound uses a classical theorem by Youngs (Theorem 4.3 in \cite{You}) which says that every minimal genus embedding $\Gamma\hookrightarrow \Sigma_\gamma$ is such that $\Sigma_\gamma\backslash\Gamma$ is disjoint union of open $2$-cells. The upper bound follows from constructing an explicit embedding. Furthermore, both bounds are sharp. The sharpness of the lower bound can for instance be seen from the $n$-cube skeleton (cf. \cite{Rin1}, \cite{BH2}) and that of the upper bound can be seen from the set of trees.

The main consequence of Proposition \ref{prp_general} that we are interested in is that if we have a sequence of graphs $\rij{\Gamma_n}{n\in\mathbb{N}}{}$ such that:
\begin{equation*}
\frac{p(\Gamma_n)}{q(\Gamma_n)}\rightarrow 0\text{ for }n\rightarrow\infty
\end{equation*}
and $h(\Gamma_n)\geq h\geq 3$ for all $n\in\mathbb{N}$ then:
\begin{equation*}
\left(\frac{1}{2}-\frac{1}{h}\right) q(\Gamma_n) \lesssim \gamma(\Gamma_n) \lesssim \frac{1}{2}q(\Gamma_n) \text{ for }n\rightarrow\infty
\end{equation*}
In other words, just by knowing the number of edges and vertices, we can determine the asymptotic behaviour of the genus of our sequence up to a multiplicative constant.

\subsection{The curve graph} 
In what follows $\Sigma_{g,n}$ will denote an orientable topological surface of genus $g$ with $n$ punctures. We set $\Sigma_g:=\Sigma_{g,0}$. 

A pants decomposition $P$ of $\Sigma_{g,n}$ is a set of isotopy classes of curves on $\Sigma_{g,n}$ such that $\Sigma_{g,n}\backslash P$ is a disjoint union of three-holed spheres (pairs of pants). The complexity $\kappa(\Sigma_{g,n})$ of $\Sigma_{g,n}$ is the number of curves in a pants decomposition of $\Sigma_{g,n}$. Depending on $g$ and $n$ there may not be any pants decompositions so we will suppose that 
$$
\kappa(\Sigma_{g,n}) = 3g-3 +n >0.
$$
When $\kappa(\Sigma_{g,n}) \geq 2$ we can define the curve graph of $\Sigma_{g,n}$ as follows:

\begin{dff} The curve graph $\CG{\Sigma_{g,n}}$ is the graph with:
\vspace{-0.1in}
\begin{itemize}[leftmargin=0.7in]
\item[vertices:] isotopy classes of simple closed curves on $\Sigma_{g,n}$.
\item[edges:] vertices $\alpha$ and $\beta$ share an edge if and only if they can be realized disjointly on $\Sigma_{g,n}$.
\end{itemize}
\end{dff}
(When $\kappa(\Sigma_{g,n}) = 1$ the curve graph is defined analogously to the pants graph which we define in the sequel.) 

We will also be interested in the modular versions of the above graphs. These are the graphs where only the homeomorphism type of the of the objects matter. Specifically:
\begin{dff} The modular curve graph $\MC{\Sigma_{g,n}}$ is the graph with:
\vspace{-0.1in}
\begin{itemize}[leftmargin=0.7in]
\item[vertices:] homeomorphism types of simple closed curves on $\Sigma_{g,n}$.
\item[edges:] vertices $\alpha$ and $\beta$ share an edge if and only if they can be realized disjointly on $\Sigma_{g,n}$.
\end{itemize}
\end{dff}
One can obtain $\MC{\Sigma_{g,n}}$ as a quotient of $\CG{\Sigma_{g,n}}$ by the action of the (extended) mapping class group $\MCG{\Sigma_{g,n}}$ defined as follows:
\begin{equation*}
\MCG{\Sigma_{g,n}} = \frac{\mathrm{Homeo}\left(\Sigma_{g,n}\right)}{\mathrm{Homeo}_0\left(\Sigma_{g,n}\right)}
\end{equation*}
In particular we note that we are allowing orientation reversing homeomorphisms and homeomorphisms that permute punctures. Except for in a few low complexity cases, the automorphism group of the curve graph is exactly the (extended) mapping class group so this amounts to looking at the curve graph up to simplicial automorphism. 

For $\Sigma_g$, the graph $\MC{\Sigma_{g}}$ is particularly simple. There are exactly 
$$
\floor{ \frac{g}{2}}  +1
$$
homeomorphism types of simple closed curves on a closed surface of genus $g$: there is one non-separating type and $\floor{ \frac{g}{2}} $ separating types depending on how much genus they leave on either side (see Figure \ref{fig:curvetypes} for examples in genus $4$ and $5$). 

\begin{figure}[H]
\begin{center} 
\includegraphics[width=12cm]{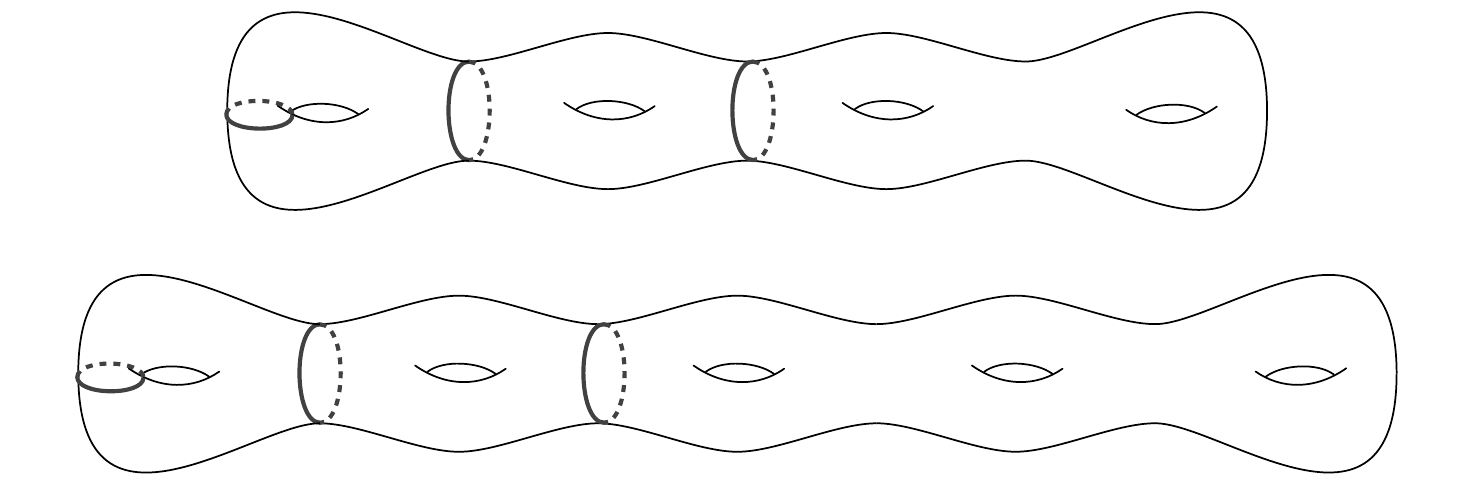} 
\caption{Topological types of curves on surfaces of genus $4$ and $5$}
\label{fig:curvetypes}
\end{center}
\end{figure}

It is not difficult to see that they can be realized disjointly so $\MC{\Sigma_{g}}$ contains a complete graph on $\floor{ \frac{g}{2}}+1$ vertices. In addition there are loops at vertices corresponding to curves that can be realized by two different disjoint isotopy class of curve simultaneously. For odd $g$ this is every curve and for even genus this is every curve except for the separating curve that separates the surface into two parts of genus $\frac{g}{2}$. So the modular curve graph is quite simple and computing its genus is paramount to Ringel and Youngs' computation of the genus of a complete graph.

\subsection{The pants graph} 
On a $\Sigma_{g,n}$ with $\kappa(\Sigma_{g,n}) \geq 1$, the pants graph $\PG{\Sigma_{g,n}}$ is a graph on isotopy classes of pants decompositions of $\Sigma_{g,n}$. An {\it elementary move} on an (internal) curve in a pants decomposition is one of the following:
\vspace{-0.1in}
\begin{itemize}[leftmargin=0.2in]
\item[-] If the curve is a boundary curve of two distinct pairs of pants then an elementary move on this curve consists of replacing it with a curve that intersects it twice.
\item[-] If the curve is a boundary curve of one pair of pants then an elementary move means replacing it with a curve that intersects it once.
\end{itemize}

The figures below illustrate the types of elementary move. Elementary moves 1 and 2 correspond to the first type and elementary move 3 to the second type. 

\begin{figure}[H]
\begin{center} 
\includegraphics[scale=1]{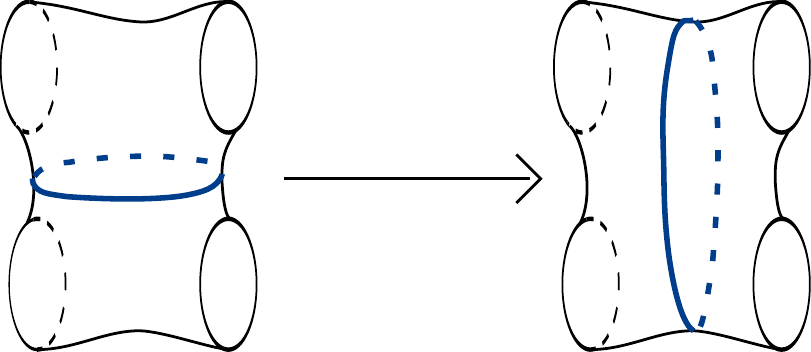} 
\caption{Elementary move 1. The black closed curves are either part of the pants decomposition or boundary components.}
\label{pic2}
\end{center}
\end{figure}

\begin{figure}[H]
\begin{center} 
\includegraphics[scale=1]{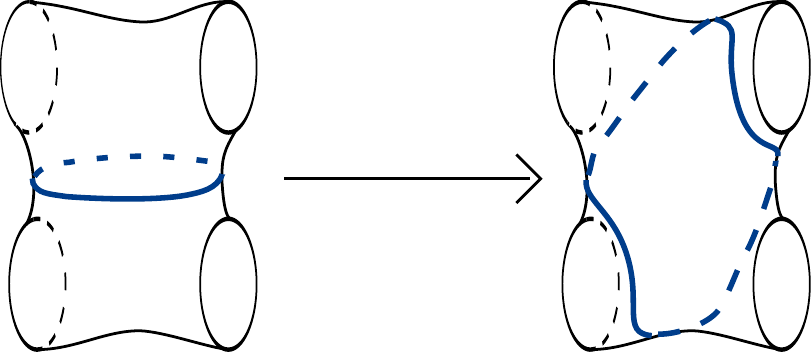} 
\caption{Elementary move 2. The black closed curves are either part of the pants decomposition or boundary components.}
\label{pic3}
\end{center}
\end{figure}

\begin{figure}[H]
\begin{center} 
\includegraphics[scale=1]{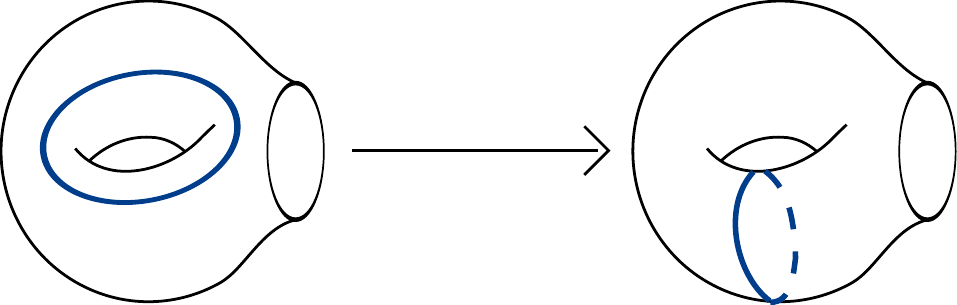} 
\caption{Elementary move 3. The black closed curve is either part of the pants decomposition or a boundary component.}
\label{pic4}
\end{center}
\end{figure}

If we start with a pants decomposition and perform an elementary move on one of its curves, the resulting set of curves will again be a pants decomposition of the surface. Elementary moves will allow us to define edges in the pants graph which we formally define as follows.

\begin{dff} The pants graph $\PG{\Sigma_{g,n}}$ is the graph with:
\vspace{-0.1in}
\begin{itemize}[leftmargin=0.7in]
\item[vertices:] pants decompositions of $\Sigma_{g,n}$.
\item[edges:] vertices $P$ and $P'$ share an edge if $P'$ can be obtained from $P$ by performing an elementary move on one of its curves.
\end{itemize}
\end{dff}

As before, we define the modular pants graph as the graph defined on these objects up to homeomorphism.
\begin{dff} The moduli pants graph $\MP{\Sigma_{g,n}}$ is the graph with:
\vspace{-0.1in}
\begin{itemize}[leftmargin=0.7in]
\item[vertices:] homeomorphism types of pants decompositions of $\Sigma_{g,n}$.
\item[edges:] vertices $P$ and $P'$ share an edge if and only if $P'$ can be obtained from $P$ by performing an elementary move on one of its curves.
\end{itemize}
\end{dff}

We again have an action of $\MCG{\Sigma_{g,n}}\curvearrowright\PG{\Sigma_{g,n}}$ through the action on curves and 
\begin{equation*}
\MP{\Sigma_{g,n}} = \frac{\PG{\Sigma_{g,n}}}{\MCG{\Sigma_{g,n}}}
\end{equation*}
As before, except for some complexity cases, the mapping class group is the automorphism group of the pants graph \cite{Mar}. 

\subsection{The flip graph} 
The final type of graphs we will be interested in is flip graphs. These are graphs of (isotopy classes) of triangulations of $\Sigma_{g,n}$ with vertices that lie in the punctures (and thus we suppose $n>0$). More precisely, by a triangulation we mean a maximal collection of isotopy classes of arcs, disjoint in their interior. These maximal multi-arcs cut the surface into triangles, but are not necessarily proper triangulations as for instance a triangle may be glued to itself. (We refer to arcs of the triangulation as opposed to edges as to distinguish them from the edges of the flip graph itself.)

Given such a triangulation, a flip in one of the arcs of the triangulation consists of replacing this edge with its `opposite diagonal'. Phrased otherwise, a flip is the removal of an arc and its replacement with the only other possible arc that can complete the resulting multi-arc into a triangulation. Note that if an arc only belongs to one triangle, it is impossible to flip it. A flip is illustrated in Figure \ref{pic8}.

\begin{figure}[H]
\begin{center} 
\includegraphics[scale=0.5]{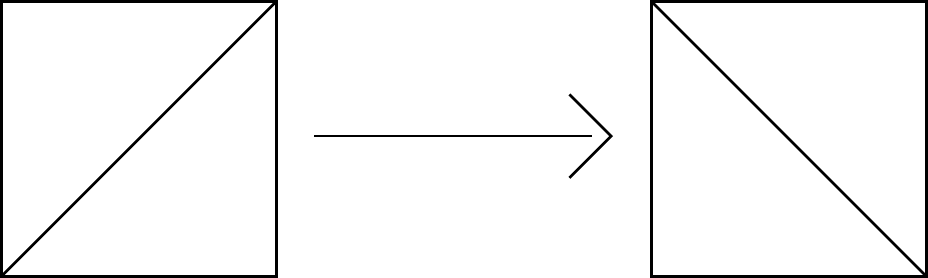} 
\caption{A flip.}
\label{pic8}
\end{center}
\end{figure}

We can now define the flip graph:
\begin{dff} Let $n>0$. The flip graph $\FG{\Sigma_{g,n}}$ is the graph with:
\vspace{-0.1in}
\begin{itemize}[leftmargin=0.7in]
\item[vertices:] isotopy types of triangulations of $\Sigma_{g,n}$ with vertices in the punctures.
\item[edges:] vertices $T$ and $T'$ share an edge if and only if $T'$ can be obtained from $T$ by performing a flip on one of the arcs of $T$.
\end{itemize}
\end{dff}
The modular version of the flip graph is defined as follows.
\begin{dff} Let $n>0$.The modular flip graph $\MF{\Sigma_{g,n}}$ is the graph with:
\vspace{-0.1in}
\begin{itemize}[leftmargin=0.7in]
\item[vertices:] homeomorphism types of triangulations of $\Sigma_{g,n}$.
\item[edges:] vertices $T$ and $T'$ share an edge if and only if $T'$ can be obtained from $T$ by performing a flip on one of the edges of $T$.
\end{itemize}
\end{dff}
Again, it is the quotient of the flip graph by the mapping class group or, except for in some low complexity cases, the quotient of the flip graph by its automorphisms \cite{KP}.

\section{Random regular graphs}\label{sec_randgr}

Both pants decompositions and triangulations can be described using cubic (or trivalent) graphs. This means that we can use results from the theory of random regular graphs to count certain `problematic' pants decompositions and triangulations. In this section we will recall the results we will use from this theory, we will restrict to the cubic case.

We begin with the definition of the set of random cubic graphs:
\begin{dff} Let $N\in 2\mathbb{N}$ then we define the set:
$$
\Omega_N=\left\{\text{partitions of }\{1,2,\ldots,3N\}\text{ into pairs} \right\}
$$
\end{dff}
An element $\omega\in\Omega_N$ corresponds to a cubic graph $\Gamma(\omega)$ with vertex set $\{1,\ldots,N\}$ in the following way: for every element in $k\in\{1,\ldots,3N\}$ we create a half edge with label $k$. We group these half edges in the sets $\{1,2,3\}$, $\{4,5,6\}$,..., $\{3N-2,3N-1,3N\}$ corresponding to vertices $1$, $2$,...,$N$. Gluing these half edges into edges according to $\omega$ gives $\Gamma(\omega)$.

To turn $\Omega_N$ into a probability space we define a probability measure:
\begin{equation*}
\mathbb{P}_N : 2^{\Omega_N} \rightarrow \mathbb{R}
\end{equation*}
by:
\begin{equation*}
\Pro{N}{A} = \frac{\aant{A}}{\aant{\Omega_N}} \text{ for all }A\subset\Omega_N
\end{equation*}
where $2^{\Omega_N}$ denotes the power set of $\Omega_N$.

\subsection{Counting cubic graphs}

The first result we will need does not directly sound like a result on random graphs, but it follows from counting the number of labelled cubic graphs, which was done by Bender and Canfield in \cite{BC} and then considering the probability that such a graph carries a non-trivial automorphism, which was done independently by Bollob\'as in \cite{Bol2} and McKay and Wormald in \cite{MW}. A simple graph is a graph without loops and without multiple edges.
\begin{thm}\label{thm_BC} \cite{BC}, \cite{Bol2}, \cite{MW} Let $I_N^*$ denote the number of isomorphism classes of simple cubic graphs on $N$ vertices. Then:
$$
I_N^* \sim \frac{1}{e^2\sqrt{\pi N}} \left(\frac{3N}{4e}\right)^{N/2}
$$
for $N\rightarrow\infty$.
\end{thm}

We will also use the result on automorphisms, so we state it separately. $\mathrm{Aut}(\Gamma)$ will denote the automorphism group of the graph $\Gamma$.
\begin{thm}\label{thm_aut} \cite{Bol2}, \cite{MW} We have:
$$
\lim\limits_{N\rightarrow\infty}\Pro{N}{\mathrm{Aut}\neq \{e\}} \; = \; 0
$$
\end{thm}

We will also need to control is the distribution of the number of circuits of $k$ edges (or $k$-circuits, where by a circuit we mean a cycle that traverses every edge at most once). Note that a $1$-circuit is a loop and a $2$-circuit is a multiple edge. We begin by defining the corresponding set of random variables. We let:
\begin{equation*}
X_{N,k}:\Omega_N\rightarrow \mathbb{N}
\end{equation*}
denote the random variable that counts the number of $k$-circuits for all $k\in\mathbb{N}$.

We will use the following theorem by Bollob\'as:
\begin{thm}\label{thm_bol}\cite{Bol1} Let $m\in\mathbb{N}$. Then:
$$
X_{N,i} \rightarrow X_i \text{ in distribution for }N\rightarrow\infty\text{ for all }i=1,\ldots, m
$$
where:
\vspace{-0.1in}
\begin{itemize}[leftmargin=0.2in]
\item[-] $X_i$ is a Poisson distributed random variable with mean $\lambda_i = 2^i/2i$.
\item[-] The random variables $X_1,\ldots,X_m$ are mutually independent.
\end{itemize}
\end{thm}

In the proof of the theorem above one uses the fact that fixed graphs that are not circuits appear with probability tending to $0$. We will need this fact and state it as a theorem:
\begin{thm}\label{thm_noncyc}\cite{Bol1} Let $H$ be a graph that contains more edges than vertices. Then:
$$
\ExV{N}{\text{number of copies of }H\text{ in }\Gamma} = \mathcal{O}\left(N^{-1}\right)
$$
for $N\rightarrow\infty$.
\end{thm} 

Finally, most graphs we will consider will actually be multigraphs, which means that we need to know the cardinality of the set of multigraphs with a fixed number of vertices as well. We have the following theorem, of which we offer a new proof below:
\begin{thm}\cite{Wor2}\label{thm_multigraphs1} Let $I_N$ denote the number of isomorphism classes of cubic multigraphs on $N$ vertices. Then:
$$
I_N \sim \frac{e^{2}}{\sqrt{\pi N}} \left(\frac{3N}{4e}\right)^{N/2}
$$
for $N\rightarrow\infty$.
\end{thm}

\begin{proof} To prove this, we define the following two sets:
\begin{equation*}
\mathcal{G}_{N}=\left\{\text{Cubic multigraphs with vertex set }\{1,\ldots,N\} \right\}
\end{equation*}
\begin{equation*}
\mathcal{U}_{N}=\left\{\text{Isomorphism classes of cubic multigraphs on }N\text{ vertices}\right\}
\end{equation*}
We have $I_N=\aant{\mathcal{U}_N}$. We will use the two natural forgetful maps:
\begin{equation*}
\Omega_{N}\xrightarrow{\pi_1}\mathcal{G}_{N}\xrightarrow{\pi_2}\mathcal{U}_{N}
\end{equation*}
If $\Gamma\in\mathcal{U}_{N}$ and $G\in\mathcal{G}_{N}$ have $k$ $1$-circuits and $l$ $2$-circuits, then we have:
\begin{equation*}
\aant{\pi_1^{-1}(G)} = \frac{6^{N}}{2^k2^l} \text{ and } \aant{\pi_2^{-1}(\Gamma)} = \frac{N!}{\aant{\mathrm{Aut}(\Gamma)}}
\end{equation*}

We know $\aant{\Omega_N}$. We will first count $\aant{\mathcal{G}_N}$ using $\pi_1$ and after that we will use $\pi_2$ to count $I_N$.

Because $X_1(\omega)\leq N$ and $X_2(\omega)\leq N$ for all $\omega\in\Omega_N$ we have:
\begin{align*}
\aant{\mathcal{G}_N} & = \frac{\aant{\Omega_N}}{6^N}\sum_{k,l=0}^N 2^k2^l\Pro{N}{X_1=k,\;X_2=l} \\
 & = \frac{\aant{\Omega_N}}{6^N}\ExV{N}{2^{X_1+X_2}} \\
 & = \frac{\aant{\Omega_N}}{6^N}\ExV{N}{\sum_{k=0}^{X_1+X_2}\binom{X_1+X_2}{k}} \\
\end{align*}
Because $X_1+X_2\leq 2N$ we have that:
\begin{equation*}
\binom{X_1+X_2}{k} = 0 \text{ for all }k> 2N
\end{equation*}
So we obtain:
\begin{align*}
\aant{\mathcal{G}_N} & = \frac{\aant{\Omega_N}}{6^N}\ExV{N}{\sum_{k=0}^{2N}\binom{X_1+X_2}{k}} \\
 & = \frac{\aant{\Omega_N}}{6^N}\sum_{k=0}^{2N}\frac{1}{k!}\ExV{N}{\left(X_1+X_2\right)_k} \\
\end{align*}
where $\left(X_1+X_2\right)_k=(X_1+X_2)(X_1+X_2-1)\cdots (X_1+X_2-k+1)$. This number of equal to the number of ordered $k$-tuples of $1$- and $2$-circuits. We define the set:
\begin{equation*}
\mathcal{C}_k = \left\{k \text{-tuples of }1\text{- and }2\text{-circuits, with half edges labelled with labels in } \{1,\ldots,3N\} \right\}
\end{equation*}
and:
\begin{equation*}
\mathcal{C}_{k,i} = \verz{c\in\mathcal{C}_k}{c\text{ contains }i\text{ }1\text{-circuits}}
\end{equation*}
So we get:
\begin{equation*}
\aant{\mathcal{G}_N} = \frac{\aant{\Omega_N}}{6^N}\sum_{k=0}^{2N}\frac{1}{k!}\cdot \frac{1}{\aant{\Omega_N}}\sum_{\omega\in\Omega_N}\sum_{i=0}^k\sum_{c\in \mathcal{C}_{k,i}} \chi_{c}(\omega)
\end{equation*}
where:
\begin{equation*}
\chi_c(\omega)= \left\{ \begin{array}{ll} 1 & \text{if } c\subset \omega \text{ as partitions}\\
 0 & \text{otherwise} \end{array}\right.
\end{equation*}
So:
\begin{equation*}
\aant{\mathcal{G}_N} = \frac{\aant{\Omega_N}}{6^N}\sum_{k=0}^{2N}\frac{1}{k!} \sum_{i=0}^k\sum_{c\in \mathcal{C}_{k,i}} \Pro{N}{c\subset\omega}
\end{equation*}
Because $\Pro{N}{c\subset\omega}$ depends only on the number of edges in $c$, it is constant on $\mathcal{C}_{k,i}$. So we fix $c_{k,i}\in \mathcal{C}_{k,i}$ and write:
\begin{align*}
\aant{\mathcal{G}_N} & = \frac{\aant{\Omega_N}}{6^N}\sum_{k=0}^{2N}\frac{1}{k!} \sum_{i=0}^k \aant{\mathcal{C}_{k,i}}\cdot \Pro{N}{c_{k,i}\subset\omega} \\
& = \frac{\aant{\Omega_N}}{6^N}\sum_{k=0}^{2N}\frac{1}{k!} \sum_{i=0}^k \frac{k! \cdot 2^{k-i} \cdot 3^{2k-i} }{i!(k-i)!\cdot 2^{k-i}}\frac{N(N-1)\cdots (N-(i+2(k-i))+1)}{(3N-1)(3N-3)\cdots (3N-2(i+2(k-i))+1)} 
\end{align*}
We have:
\begin{equation*}
\frac{N-m}{3N-2m-1} \leq \frac{N}{3N-1} \text{ for all }m=0,\ldots 2k-i
\end{equation*}
So, using dominated convergence, we obtain:
\begin{align*}
\aant{\mathcal{G}_N} &\sim \frac{\aant{\Omega_N}}{6^N}\sum_{k=0}^{2N}\frac{1}{k!} \sum_{i=0}^k \binom{k}{i} \cdot 3^{2k-i} \cdot \frac{1}{3^{2k-i}} \\
&= \frac{\aant{\Omega_N}}{6^N}\sum_{k=0}^{2N}\frac{1}{k!} 2^k \\
&\sim e^{2 }\frac{\aant{\Omega_N}}{6^N}
\end{align*}
for $N\rightarrow\infty $.

We will now use this, combined with $\pi_2$ to compute $I_N$. In particular, we will show that we can ignore graphs with automorphisms. We have:
\begin{align*}
\frac{\aant{\verz{G\in \mathcal{G}_N}{\mathrm{Aut}(G)\neq \{e\}}}}{\aant{\mathcal{G}_N}} & = \frac{\sum\limits_{k,l}2^{k+l}\Pro{N}{X_1=k,X_2=l,\mathrm{Aut}\neq \{e\}}}{6^N}\frac{\aant{\Omega_N}}{\aant{\mathcal{G_N}}}
\end{align*}
Theorem \ref{thm_bol} tells us that there exists a $C>0$ such that for all $k,l\in\mathbb{N}$:
\begin{equation*}
2^{k+l}\Pro{N}{X_1=k,X_2=l,\mathrm{Aut}\neq \{e\}}\; \leq\; 2^{k+l}\Pro{N}{X_1=k,X_2=l}\; \leq\; C\cdot 2^{k+l}\frac{\lambda_1^{k} e^{-\lambda_1}\lambda_2^{l} e^{-\lambda_2}}{k!\; l!}
\end{equation*}
which is summable and does not depend on $N$. Hence by the dominated convergence theorem we can take the limits of the terms to compute the limit of the sum. We have:
\begin{equation*}
\lim\limits_{N\rightarrow\infty }\Pro{N}{X_1=k,X_2=l,\mathrm{Aut}\neq \{e\}} \; \leq \;\lim\limits_{N\rightarrow\infty }\Pro{N}{\mathrm{Aut}\neq \{e\}} \;= \; 0
\end{equation*}
by Theorem \ref{thm_aut}. So we obtain that: 
\begin{equation*}
\lim\limits_{N\rightarrow\infty}\frac{\aant{\verz{G\in \mathcal{G}_N}{\mathrm{Aut}(G)\neq \{e\}}}}{\aant{\mathcal{G}_N}} = 0
\end{equation*}
and:
\begin{equation*}
\frac{\aant{\mathcal{G}_N}}{\aant{\mathcal{U}_N}} \sim N!
\end{equation*}
for $N\rightarrow\infty$. Combining the above and applying Stirling's approximation now gives the theorem.
\end{proof}

\subsection{Maps with a small defect}

We will also need to control the number of graphs that carry a map which distorts the adjacency structure at at most a fixed number of edges. This is a slight generalization of an automorphism of a graph. In Theorem \ref{thm_aut} we have already seen that the probability that a random cubic graph carries a non-trivial automorphism tends to $0$ for $n\rightarrow\infty$ (cf. also \cite{KSV}, \cite{Wor}). In particular, in \cite{KSV}, Kim, Sudakov and Vu also consider maps of small distortion, but for regular graphs with growing vertex degrees.

We note that the set of bijections of the vertex set $V$ of a graph can be identified with the symmetric group $\symm{V}$. The distortion we were speaking about is the edge defect defined below.
\begin{dff} Let $\Gamma=(V,E)$ be a graph and $\pi\in \symm{V}$. The edge defect of $\pi$ on $\Gamma$ is the number:
$$
\mathrm{ED}_\pi(\Gamma) = \aant{\verz{e\in E}{\pi(e)\notin E}}
$$
\end{dff}
Note the difference with the similar definition of the defect of a permutation by Kim, Sudakov and Vu in \cite{KSV}.

We will also need to consider the action of an element $\pi\in\symm{n}$ on the edges $K_n$, the complete graph on $n$ vertices. From hereon an edge orbit of an element $\pi\in\symm{n}$ will mean the orbit of an edge in $K_n$ under $\pi$.

We want to bound the probability that a cubic graph carries a non-trivial map with edge defect $\leq k$ for a fixed $k\in\mathbb{N}$. To do this, we will adapt the proof of Wormald in \cite{Wor} of the fact that a random regular graph asymptotically almost surely carries no automorphisms to our situation. The key ingredient is the following lemma (Equations 2.3 and 2.7 in \cite{Wor}):
\begin{lem}\label{lem_wor} \cite{Wor} There exists a constant $C>0$ such that: given $N\in\mathbb{N}$, $a\in\mathbb{N}$, $s_i\in\mathbb{N}$ for $i=2,\ldots,6$ such that $2s_2+3s_3+\ldots 6s_6\leq N$ and $e_1\in\mathbb{N}$ such that $e_1\leq s_2$ and $f\in\mathbb{N}$. Let $\mathcal{H}_N(a,s_1,\ldots, s_6,e_1,f,k)$ be the set of pairs $(\pi,H)$ such that:
\vspace{-0.1in}
\begin{itemize}[leftmargin=0.2in]
\item[-] $\pi\in \symm{N}$, $H$ a graph on $\{1,\ldots,N\}$.
\item[-] $\pi\neq \mathrm{id}$ and $\pi$ has $s_i$ $i$-cycles for $i=2,\ldots 6$.
\item[-] The support of $\pi$ is $A$, $\aant{A}=a$ and $\pi$ fixes $H$ as a graph.
\item[-] $\deg_H(x) = 3$ for all $x\in A$.
\item[-] At least one end of every edge in $H$ is moved by $\pi$.
\item[-] The subgraph of $H$ induced by $A$ can be written as the union of $f$ edge-orbits of $\pi$ and $f$ is minimal in this respect.
\item[-] $e_1$ edges of $H$ are fixed edge-wise by $\pi$.
\item[-] The subgraph of $H$ induced by $A$ has $k$ edges.
\end{itemize}
Then:
$$
\sum_{(\pi,H)\in \mathcal{H}_N(a,s_1,\ldots, s_6,e_1,f)}\Pro{N}{H\subset\Gamma} \leq C^a N^{-a/2}a^{3a/8}
$$
\end{lem}

Using this, we can prove the following, where we write $a(\pi)$ for the number of elements in the support $A(\pi)$ of $\pi\in\symm{2N}$:
\begin{prp}\label{prp_symm1} Let $n,k\in\mathbb{N}$. There exists a $C>0$ such that:
$$
\Pro{N}{\substack{\displaystyle{\exists\; \mathrm{id}\neq \pi_1,\ldots,\pi_n\in \symm{2N}\text{ such that } \mathrm{ED}_{\pi_i}(\Gamma)\leq k \text{ and }a(\pi_i) \geq k-1} \\  \displaystyle{\forall i=1,\ldots,n \text{ and }\Gamma\text{ has }< Cn\text{ circuits of length }\leq k}}}\rightarrow 0
$$
for $N\rightarrow\infty$.
\end{prp}

\begin{proof} The idea of the proof is as follows: we consider a set of graphs $\mathcal{H}'_N$ such that any graph with no circuits of less than $k$ edges that carries a map $\pi\in\symm{N}$ with edge defect $\leq k$ contains at least one graph in the set $\mathcal{H}'_N$ as a subgraph. If we can then prove that
\begin{equation*}
\sum\limits_{H'\in\mathcal{H}'_N} \Pro{N}{H'\subset\Gamma} \rightarrow 0
\end{equation*} 
for $N\rightarrow\infty$, then we have proved the proposition. Indeed every circuit adds at most a finite number of maps with a bounded edge defect; this is where the constant $C$ comes from.

We will construct a part of $\mathcal{H}'_N$ out of the set $\mathcal{H}_N$ which is the set of graphs of which at least one must be a subgraph of a graph with a non-trivial automorphism. This last set is given by:
\begin{equation*}
\mathcal{H}_N=\bigcup\limits_{\substack{\pi\in\symm{n} \\ \pi\neq\mathrm{id}}} \mathcal{H}_N(\pi)
\end{equation*}
where:
\begin{equation*}
\mathcal{H}_N(\pi)  = \verz{H\text{ graph on } \{1,\ldots,N\}}{\substack{\displaystyle{\text{every edge in }H\text{ has at least on end in }A(\pi) }\\ \displaystyle{\deg_H(x)=3 \; \forall x\in A} \\ \displaystyle{\pi H = H}}}
\end{equation*}
The condition $\pi H=H$ is equivalent to the fact that $H$ can be written as a union of edge-orbits of $\pi$.

We are interested in graphs that carry a non-trivial `almost-automorphism'. Suppose we have a graph $\Gamma$ for which the map $\pi\in \symm{n}$ with support $A(\pi)\subset \{1,\ldots,N\}$ is such an almost-automorphism. We consider the graph $H\subset\Gamma$ that consists of all edges that have at least one end in $A(\pi)$. All but $k$ images of the edges of $H$ (seen as a subgraph of $K_N$) under $\pi$ should be edges in $\Gamma$ again. This means that $H$ can be written as a union of edge orbits of $\pi$ out of which $k$ edges have been removed and replaced by different edges. We are not interested in these replacement edges and consider the graph $H'$ that consists of $H$ minus these edges. We define:
\begin{equation*}
\mathcal{H}'_N(\pi) = \verz{H'\text{ graph on }\{1,\ldots, N\}}{\;\substack{\displaystyle{\exists H\in\mathcal{H}_N(\pi)\text{ such that }H\backslash H'\text{ has }k\text{ edges }} \\ \displaystyle{H' \text{ contains no circuits of length }\leq k}}}
\end{equation*}
What we have shown is that for $\pi\in\symm{N}$:
\begin{equation*}
\Gamma \text{ has no circuits of length }\leq k\text{ and }\mathrm{ED}_\pi(\Gamma) \leq k \;\Rightarrow \;\exists H'\in\mathcal{H}'_N(\pi) \text{ with }H'\subset\Gamma
\end{equation*}

We now want to apply Lemma \ref{lem_wor}. This means that we need relate the cardinalities of $\mathcal{H}_N(\pi)$ and $\mathcal{H}'_N(\pi)$ and the probabilities $\Pro{N}{H\subset\Gamma}$ and $\Pro{N}{H'\subset\Gamma}$, where $H'$ is the graph obtained from $H\in\mathcal{H}_N(\pi)$ by removing two edges.

Because a graph in $\mathcal{H}_N(\pi)$ has at most $3a/2$ edges we get:
\begin{equation*}
\aant{\mathcal{H}'(\pi)} \leq \frac{1}{2}\left(\frac{3a}{2}\right)^k \cdot \aant{\mathcal{H}_N(\pi)}
\end{equation*}
and because $H'$ contains $k$ edges fewer than $H$, we have:
\begin{equation*}
\Pro{N}{H'\subset\Gamma} \leq C'N^k \cdot \Pro{N}{H\subset\Gamma}
\end{equation*}
where $C'>0$ is independent of $H$,$H'$ and $N$.

It turns out that the bounds above in combination with Lemma \ref{lem_wor} are only small enough when the support of $\pi$ is large enough, i.e. when it contains at least $2k+1$ elements. This means that we need to cut the sum over subgraphs into two pieces. Recall that $a(\pi)$ denotes the number of elements in the support $A(\pi)$ of $\pi\in\symm{N}$ and define:
\begin{equation*}
T_1(N) = \sum\limits_{\substack{\pi\in\symm{N} \\ k-1\leq a(\pi)\leq 2k, \pi\neq \mathrm{id}}} \sum\limits_{H'\in\mathcal{H}'_N(\pi)} \Pro{N}{H'\subset\Gamma}
\end{equation*}
\begin{equation*}
T_2(N) = \sum\limits_{\substack{\pi\in\symm{n} \\ a(\pi)> 2k}} \sum\limits_{H'\in\mathcal{H}'_N(\pi)} \Pro{N}{H'\subset\Gamma}
\end{equation*}
So now we need to prove that both $T_1(N)$ and $T_2(N)$ tend to $0$ for $N\rightarrow\infty$.

We start with $T_2(N)$. We have:
\begin{align*}
T_2(N) & \leq  \sum\limits_{\substack{\pi\in\symm{N} \\ a(\pi)> 2k}}  C''\cdot a(\pi)^2 N^k \sum\limits_{H\in\mathcal{H}_N(\pi)}\Pro{N}{H\subset\Gamma} \\
 & \leq  \sum\limits_{\substack{a,s_1,\ldots,s_6 \\ e_1,f}}  C''\cdot a^2 N^k \sum_{(\pi,H)\in \mathcal{H}_N(a,s_1,\ldots, s_6,e_1,f)}\Pro{N}{H\subset\Gamma} \\
 & \leq  \sum\limits_{a=2k+1}^N  (C''')^a N^{-a/2+k}
\end{align*}
for constants $C'',C'''>0$ independent of $a$ and $N$ (we used Lemma \ref{lem_wor} and the fact that the number of choices for the variables other than $a$ is polynomial in $a$ in the last step). The final expression tends to $0$ for $N\rightarrow\infty$.

To prove that $T_1(N)$ also tends to $0$ we will need to use the assumptions on short circuits and the support of the supposed almost-automorphisms. Recall that we are summing over all permutations $\pi\in\symm{N}$ with $k-1\leq a(\pi)\leq 2k$. We first note that the set of isomorphism classes of graphs we are summing over is finite (they are graphs of bounded degree on at most $2k$ vertices). Because none of these are circuits by assumption, Theorem \ref{thm_noncyc} tells us that if they are not subtrees, they have asymptotic probability $0$ of appearing in the graph. 

As such the only subgraphs we need to worry about are small subtrees. Because we assume that the support of $\pi$ contains at least $k-1$ vertices, this means that we need to consider subtrees of at least $k-1$ vertices. A counting argument shows that such a tree needs to be connected to the rest of the graph by at least $k+1$ edges. This leaves two options. Either $\pi$ could have edge defect $k+1$, in which case we are done, or at least two of these edges connect to the same vertex, in which case our graph needs to contain a subgraph that is either a short circuit or a more complicated graph and then we can apply the same reasoning as above.
\end{proof}

\section{Random oriented graphs and random triangulated surfaces} \label{sec_randsurf}

Random cubic graphs also naturally come with an orientation: a cyclic ordering at each vertex of the half edges emanating from this vertex. We simply declare the cyclic order to be the same as the cyclic order of the labels on the half edges. Such an orientation gives rise to a triangulated surface. We glue a triangle over every vertex of the graph and glue the sides of triangles together according to the edges of the graph, which we do in such a way that the orientation at every vertex corresponds to an outward orientation on the triangle via the right hand rule. This means that a random cubic graph gives rise to a notion of a random surface, which we will denote by $S(\omega)$ for $\omega\in\Omega_N$.

These random surfaces were for example investigated in \cite{BM}, \cite{PS} and \cite{Pet}. In particular, in \cite{Pet}, the second author developed techniques for restricting to surfaces of a fixed topology. We will use these results to study the genus of the flip graph of $\Sigma_{g,1}$. It follows from an Euler characteristic argument that in this case the number of vertices of the corresponding graph will need to be $4g-2$. 

Again, we start with a cardinality result (Theorem B of \cite{Pen2}):
\begin{thm}\label{thm_cardfg1}\cite{Pen2} We have:
$$
\aant{\verz{\omega\in\Omega_N}{S(\omega)\text{ has }1\text{ puncture}}} \sim \frac{2\sqrt{2}}{3N} \left(\frac{3N}{e}\right)^{3N/2}
$$
for $N\rightarrow\infty$.
\end{thm}

It follows from Theorem B in \cite{Pet} that:

\begin{prp} \label{prp_circ} Let $m\in\mathbb{N}$. There exists a set of mutually independent random variables $X'_i:\mathbb{N}\rightarrow\mathbb{N}$ for $i=1,\ldots,m$ such that when we restrict to surfaces with $1$ puncture:
$$
X_{N,i}\rightarrow X'_i \text{ in distribution for }N\rightarrow\infty\text{ and }i=1,\ldots, m
$$
where the limit has to be taken over $N\in 4\mathbb{N}$. 
\end{prp}

\begin{proof} This follows from the fact that the random variables are sums of a finite number of converging random variables. 
\end{proof}

It also follows from the results of \cite{Pet} (Lemma 6.3 there) that if a labelled oriented graph $H$ contains no left hand turn cycle (a cycle of the graph that turns in the direction dictated by the cyclic order at every vertex) then we still have:
\begin{equation*}
\ProC{N}{H\subset\Gamma}{S\text{ has }1\text{ punture}} \leq C N^{-\aant{E(H)}}
\end{equation*}
where $C>0$ does not depend on $H$ or $N$. This means that all the arguments in the proof of Proposition \ref{prp_symm1} still work in the case where we restrict to surfaces with $1$ puncture. So we obtain:

\begin{prp}\label{prp_symm2} Let $n,k\in\mathbb{N}$. There exists a $C>0$ such that the probability:
$$
\ProC{N}{\substack{\displaystyle{\exists\; \mathrm{id}\neq \pi_1,\ldots,\pi_n\in \symm{N}\text{ such that } \mathrm{ED}_{\pi_i}(\Gamma)\leq k }\\ \displaystyle{\text{and }a(\pi_i)\geq k-1\; \forall i=1,\ldots,n \text{ and }\Gamma} \\ \displaystyle{\text{has }< Cn\text{ circuits of length }\leq 2k}}}{S\text{ has }1\text{ punture}} \rightarrow 0
$$
as $N\rightarrow\infty$.
\end{prp}
In particular, the analogous theorem to Theorem \ref{thm_aut} also holds in this setting: asymptotically almost surely $1$-vertex triangulations have no automorphisms. 

This means that we also have the following:
\begin{thm}\label{thm_cardfg2}\cite{Pen2}. Let $I_{g,1}$ denote the number of isomorphism classes of triangulations of $\Sigma_{g,1}$. We have:
$$
I_{g,1} \sim \frac{2}{3\sqrt{\pi}\; (4g-2)^{3/2}} \left(\frac{12g-6}{e}\right)^{2g-1}
$$
for $g\rightarrow\infty$.
\end{thm}

\begin{proof} We can run the same proof as for Theorem \ref{thm_multigraphs1}. The difference however is that we do not need to worry about $1$- or $2$-circuits: there will not be any $1$-circuits and because of the orientation $2$-circuits no longer add a factor $2$. Furthermore, asymptotically there are no automorphisms. So we get:
\begin{equation*}
I_{g,1} \sim \frac{1}{3^{4g-2}(4g-2)!}\cdot \aant{\verz{\omega\in\Omega_{4g-2}}{S(\omega)\text{ has }1\text{ puncture}}}
\end{equation*}
Filling in Theorem \ref{thm_cardfg1} and using Stirling's approximation gives the result.
\end{proof}

\section{The genus of the (modular) curve graph}

For the curve graphs we will not need the random graph approach. Instead we rely on two classical theorems by Ringel and Ringel-Youngs. The first one is about the bipartite complete graph $K_{m,n}$. This is the graph with vertices $\{x_1,\ldots,x_m,y_1,\ldots,y_n\}$ and edges $\{x_i,y_j\}$ for all $i=1,\ldots m$, $j=1,\ldots,n$.

\begin{thm}\cite{Rin2} Let $m,n\in\mathbb{N}$. We have:
$$
\gamma(K_{m,n}) = \ceil{\frac{(m-2)(n-2)}{4}}
$$
\end{thm}

Using this we can compute the genus of the curve graph.
\begin{thm}\label{thm_cg} Let $g\geq 2$. We have:
$$
\gamma(\CG{\Sigma_{g}}) =\infty
$$
\end{thm}
\begin{proof} We need to prove that there exists no finite genus surface $\Sigma_{g}$. We will do this by embedding $K_{m,m}$ into  $\CG{\Sigma_{g}}$ for arbitrarily large $m$. We consider curves $\alpha_1$, $\beta_1$, $\alpha_2$ and $\beta_2$ on $\Sigma_{g,n}$ such that:
\begin{equation*}
i(\alpha_i,\beta_j)=\delta_{ij},\; i(\alpha_1,\alpha_2)=0\text{ and }i(\beta_1,\beta_2)=0
\end{equation*}
Curves with the properties of $\alpha_1$, $\beta_1$, $\alpha_2$ and $\beta_2$ exist because of the assumption on the genus. Now let $D_{\beta_i}$ denote the Dehn twist around $\beta_i$ for $i=1,2$. Then the vertices:
\begin{equation*}
\left\{\alpha_1,D_{\beta_1}\alpha_1,\ldots, D_{\beta_1}^{m-1}\alpha_1,\alpha_2,D_{\beta_2}\alpha_2,\ldots D_{\beta_2}^{m-1}\alpha_2\right\}
\end{equation*}
induce a $K_{m,m}$.
\end{proof}
We have already noted that the modular curve graph of a closed surface is a complete graph with loops added to all or all but one of the vertices. These loops do no change the genus of a graph, so we can ignore them. The genus of the complete graph is known through a classical theorem of Ringel and Youngs:
\begin{thm}\cite{RY} Let $n\in\mathbb{N}$. We have:
$$
\gamma(K_n) = \ceil{\frac{(n-3)(n-4)}{12}}
$$
\end{thm}
In the case of the modular curve graph we need to compute its number of vertices. This is given by the number of homeomorphism types of curves on $\Sigma_{g}$ which is equal to
\begin{equation*}
\floor{\frac{g}{2}}+1
\end{equation*}
Thus as a corollary to the Ringel-Youngs theorem, we obtain an exact formula for the genus of the modular curve graph. 
\begin{cor}\label{cor_mc} Let $g\in\mathbb{N}$. Then:
$$
\gamma(\MC{\Sigma_{g}}) = \ceil{\frac{(\floor{\frac{g}{2}}-2)(\floor{\frac{g}{2}}-3)}{12}}
$$
\end{cor}

\section{The genus of the (modular) pants graph}

For the pants graph we have:
\begin{thm}\label{thm_pg} Let $g\geq 2$, then:
$$
\gamma(\PG{\Sigma_g}) =\infty
$$
\end{thm}
\begin{proof} We will prove this by constructing a subgaph $Z_m\subset\PG{\Sigma_g}$ with $\gamma(Z_m)\gtrsim m^2/6$ for arbitrarily large $m$. We consider a pants decomposition $P$ of $\Sigma_g$ that contains two curves that cut off distinct one holed tori. So $P$ has two curves, which we denote $\alpha_1$ and $\alpha_2$, that are found inside these one holed tori (the curves look like those in Figure \ref{pic4}). We choose two curves $\beta_1$ and $\beta_2$ that also lie entirely in the one holed tori and that satisfy\begin{equation*}
i(\beta_i,\alpha_j)=\delta_{ij}
\end{equation*}
Let $m\in\mathbb{N}$. We look at the set of pairs of curves:
\begin{equation*}
\{\alpha_1,\beta_1,D_{\alpha_1}\beta_1,\ldots,D_{\alpha_1}^m\beta_1\}\times\{\alpha_2,\beta_2,D_{\alpha_2}\beta_2,\ldots,D_{\alpha_2}^m\beta_2\}
\end{equation*}
Each pair $(a_1,a_2)$ in this set corresponds to a pants decomposition $P(a_1,a_2)$ by replacing $\alpha_1$ in $P$ with $a_1$ and $\alpha_2$ in $P$ with $a_2$. We will identify this set of pairs of curves with the corresponding set of pants decompositions and consider the subgraph $Z_m\subset\PG{\Sigma_g}$ it induces. We have:
\begin{equation*}
p(Z_m)=(m+2)^2
\end{equation*}
Furthermore:
$$
\deg\left(\alpha_1,D_{\alpha_2}^k\beta_2\right)=\deg\left(D_{\alpha_1}^k\beta_1,\alpha_2\right)=m+4
$$
and 
$$
\deg\left(D_{\alpha_1}^k\beta_1,D_{\alpha_2}^l\beta_2\right)=6\text{ for }k,l=1,\ldots,m-1
$$
This means that:
\begin{equation*}
q(Z_m) \geq \frac{1}{2}\left(2(m-1)(m+4)+6(m-1)^2\right)\geq 4m^2-1
\end{equation*}
Finally, we have $h(Z_m)=3$. Proposition \ref{prp_general} now gives the desired result.
\end{proof}

To compute the genus of the modular pants graph, we will use the random graph tools from the previous two sections. The reason we can use these results is that pants decompositions of $\Sigma_g$ are in one to one correspondence with cubic graphs on $2g-2$ vertices. The correspondence is given by the dual graph to a pants decomposition. This is the graph with the pairs of pants as vertices and an edge between every pair of pairs of pants per curve they share. Note that this dual graph can in fact be a multigraph; it can contain loops corresponding to pairs of pants like the one in Figure \ref{pic4} and it can contain a $2$-cycle when two distinct pairs of pants share two curves. 

So, the vertices of $\MP{\Sigma_g}$ can be seen as homeomorphism types of cubic multigraphs. When counting the edges of $\MP{\Sigma_g}$, we only need to consider elementary moves that make a difference on the graph level. On the graph level the elementary moves in Figures \ref{pic2}, \ref{pic3} and \ref{pic4} look as follows:
\begin{figure}[H]
\begin{center} 
\includegraphics[scale=0.25]{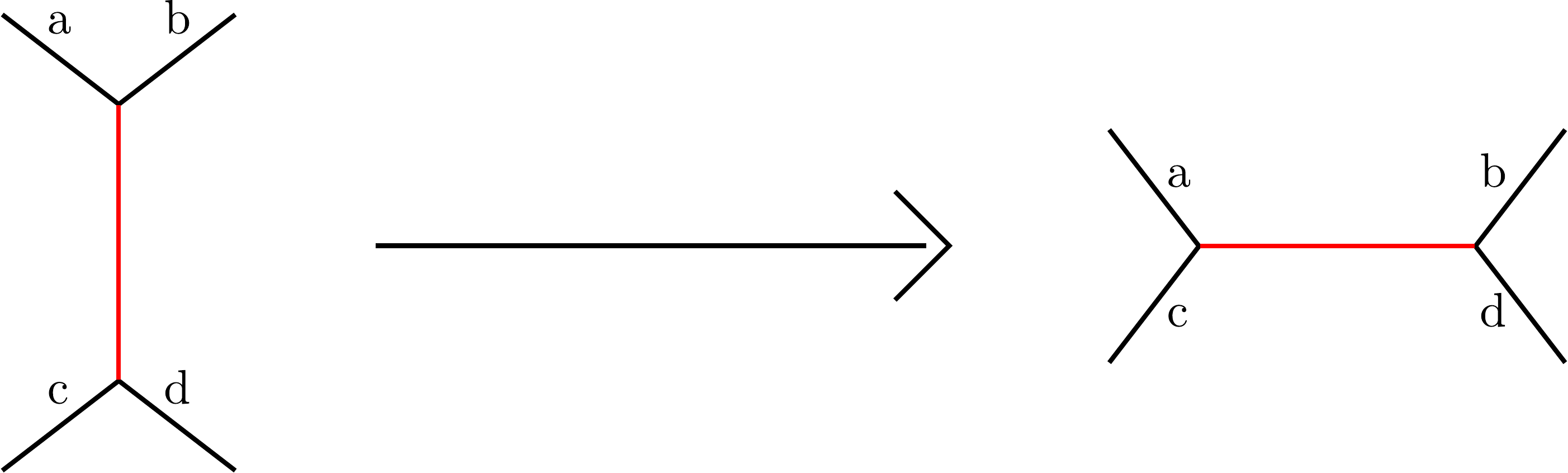} 
\caption{Elementary move 1. Half edges $a$,$b$,$c$ and $d$ represent pants curves.}
\label{pic5}
\end{center}
\end{figure}

\begin{figure}[H]
\begin{center} 
\includegraphics[scale=0.25]{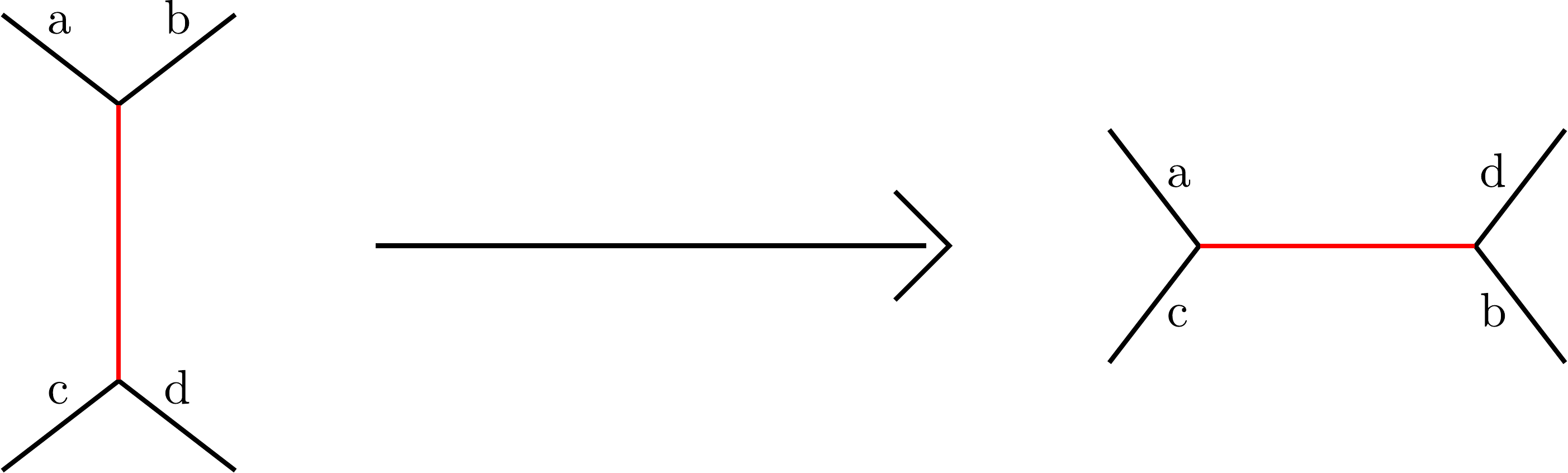} 
\caption{Elementary move 2. Half edges $a$,$b$,$c$ and $d$ represent pants curves.}
\label{pic6}
\end{center}
\end{figure}

\begin{figure}[H]
\begin{center} 
\includegraphics[scale=0.25]{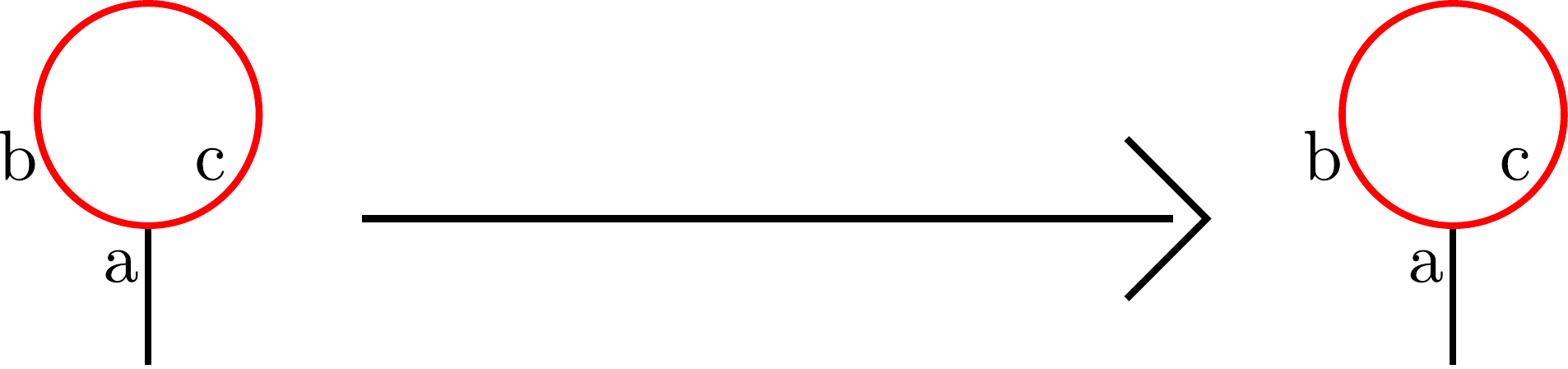} 
\caption{Elementary move 3. Half edge $a$ represents a pants curve.}
\label{pic7}
\end{center}
\end{figure}

The labels in Figures \ref{pic6} and \ref{pic7} are meant to indicate the difference between the two corresponding moves. Also note that move 3 does not do anything to the homeomorphism type of the pants decomposition. This means that the corresponding edge in $\PG{\Sigma_g}$ projects to a loop in $\MP{\Sigma_g}$.

Before we state out theorem, we define the following notation.
\begin{dff}
Let $f,g:\mathbb{N}\rightarrow\mathbb{N}$ and $c_1,c_2>0$. If
$$
\liminf_{n\rightarrow\infty} \frac{f(n)}{g(n)}\geq c_1 \text{ and }\limsup_{n\rightarrow\infty} \frac{f(n)}{g(n)}\leq c_2
$$
we write
$$
f(n)\;\sim_{c_1,c_2}\;g(n) \text{ for }n\rightarrow\infty
$$
\end{dff}

We can now state our theorem. 
\begin{thm}\label{thm_mp} We have:
$$
\gamma\left(\MP{\Sigma_g}\right)\; \sim_{c_1,c_2}\; \frac{1}{\sqrt{2g-2}}\cdot\left(\frac{6g-6}{4e} \right)^g
$$
for $g\rightarrow\infty$, where $c_1=\frac{1}{3e\sqrt{\pi}}$ and $c_2=\frac{e^{3}}{\sqrt{\pi}}$.
\end{thm}

\begin{proof} We start with the upper bound. In every pants decomposition $P$ there are $3g-3$ curves on which we can at most perform $2$ elementary moves that make a difference and are distinct on the graph level (those of Figures \ref{pic5} and \ref{pic6}). This means that in $\MP{\Sigma_g}$:
\begin{equation*}
\deg(P) \leq 6g-6
\end{equation*}
The upper bound now follows directly by applying Theorem \ref{thm_multigraphs1} and Proposition \ref{prp_general}.

The lower bound is more difficult to obtain. What we need is a lower bound on the number of edges emanating from a `generic' pants decomposition in $\MP{\Sigma_g}$. More precicely, given $m\in\mathbb{N}$ we want to understand the ratio:
\begin{equation*}
\frac{\aant{\verz{P\in\MP{\Sigma_g}}{\deg(P)\geq 6g-6-m}}}{\aant{\MP{\Sigma_g}}}
\end{equation*}

To do this we will use the results from random cubic graphs from Section \ref{sec_randgr}. We define the following two sets:
\begin{equation*}
\mathcal{G}_{2g-2}=\left\{\text{Cubic graphs with vertex set }\{1,\ldots,2g-2\} \right\}
\end{equation*}
\begin{equation*}
\mathcal{U}_{2g-2}=\left\{\text{Isomorphism classes of cubic graphs on }2g-2\text{ vertices}\right\}
\end{equation*}
$\mathcal{U}_{2g-2}$ can be identified with the vertex set of $\MP{\Sigma_g}$ and we have two forgetful maps:
\begin{equation*}
\Omega_{2g-2}\xrightarrow{\pi_1}\mathcal{G}_{2g-2}\xrightarrow{\pi_2}\mathcal{U}_{2g-2}
\end{equation*}
Furthermore, if $\Gamma\in\mathcal{U}_{2g-2}$ and $G\in\mathcal{G}_{2g-2}$ have $k$ $1$-circuits and $l$ $2$-circuits, then we have:
\begin{equation*}
\aant{\pi_1^{-1}(G)} = \frac{6^{2g-2}}{2^k2^l} \text{ and } \aant{\pi_2^{-1}(\Gamma)} = \frac{(2g-2)!}{\aant{\mathrm{Aut}(\Gamma)}}
\end{equation*}
We want to relate probabilities in $\Omega_{2g-2}$ to probabilities in $\mathcal{U}_{2g-2}$. From the two equations above we see that $1$-circuits, $2$-circuits and automorphisms cause trouble when relating these probabilities. To avoid this, we will consider only graphs in $\mathcal{U}_{2g-2}$ that don't have any such circuits. We define:
\begin{equation*}
\mathcal{U}_{2g-2}^*:=\verz{\Gamma\in\mathcal{U}_{2g-2}}{\Gamma\text{ carries no non-trivial automorphisms and has no }1\text{- and }2\text{-circuits}}
\end{equation*}

If $A\subset\mathcal{U}_{2g-2}^*$ then:
\begin{align*}
\frac{\aant{A}}{\aant{\mathcal{U}_{2g-2}^*}} & = \frac{\aant{\pi_1^{-1}\circ\pi_2^{-1}(A)}}{\aant{\pi_1^{-1}\circ\pi_2^{-1}(\mathcal{U}_{2g-2}^*)}} \\
& \geq \frac{\aant{\pi_1^{-1}\circ\pi_2^{-1}(A)}}{\aant{\Omega_{2g-2}}} \\
& = \Pro{2g-2}{\pi_1^{-1}\circ\pi_2^{-1}(A)}
\end{align*}

So what remains is to translate the statement `$\deg(P)\geq 6g-6-m$' into the language of graphs. More precisely, we need to find a necessary condition for a vertex in  $\MP{\Sigma_g}$ to have less than $6g-6-m$ edges and this condition needs to be phrased in terms of the graph dual to the pants decomposition of that vertex. There are essentially two ways in which a pants decomposition $P$ can `lose' an outgoing edge:
\vspace*{-0.1in}
\begin{itemize}[leftmargin=0.4in]
\item[(i)] $P$ shares an edge with a pants decomposition $P'$ that is isomorphic to $P$.
\item[(ii)] $P$ shares edges with two pants decompositions $P_1$ and $P_2$ that are isomorphic to each other.
\end{itemize} 
We start with situation (i). The first possibility in this case is that $P$ contains a curve on which no elementary move makes a difference in the local adjacency structure. This means that it comes from a loop in the dual graph, which we will exclude by considering $\mathcal{U}^*_{2g-2}$. If this is not the case, we proceed as follows. We label the pairs of pants in $P$ and $P'$ with the labels $1$,$2$,...,$2g-2$ consistently (i.e. the pairs of pants that are `the same' in $P$ and $P'$ get the same label) and consider the map
\begin{equation*}
F:P\rightarrow P'
\end{equation*}
that sends vertex $i$ to vertex $i$ for all $i=1,...,2g-2$. Because $P$ and $P'$ are related by an elementary move, $F$ preserves the adjacency structure everywhere except around the curve on which the move was performed. Here $F$ changes $2$ edges. By assumption we have an isomorphism $\varphi:P'\rightarrow P$. Because the adjacency structure around the curve on which the move is performed is different in $P$ and $P'$, the map $\varphi$ acts non-trivially on the labels. This means that the map:
\begin{equation*}
\varphi\circ F:P\rightarrow P
\end{equation*}
is a non-trivial element in $\symm{2g-2}$ with:
\begin{equation*}
\mathrm{ED}_{\varphi\circ F}(P) \leq 2
\end{equation*}
This is the graph theoretic description we need.

In situation (ii) we exclude loops again and we consider the maps $F_1:P\rightarrow P_1$ and $F_2:P\rightarrow P_2$ and the isomorphism $\varphi:P_1\rightarrow P_2$. The map:
\begin{equation*}
F_2^{-1}\circ\varphi\circ F_1: P \rightarrow P
\end{equation*}
is now a non-trivial map with edge defect: 
\begin{equation*}
\mathrm{ED}_{F_2^{-1}\circ\varphi\circ F_1}(P)\leq 4
\end{equation*}

Also note that the maps $\varphi\circ F$ and $F_2^{-1}\circ\varphi\circ F_1$ can not be two cycles, because they both have to move the two vertices around the edge on which the move was performed and cannot interchange these. So in particular, their supports contain at least $3$ vertices.

Every such map collapses at most $1$ of the $6g-6$ edges. So we get that:
\begin{equation*}
\Pro{2g-2}{\deg(P)\geq 6g-6-m} \geq 1-\Pro{2g-2}{\substack{\displaystyle{\exists \pi_1,\ldots,\pi_m\in\symm{2g-2}\backslash\{\mathrm{id}\}}\\ \displaystyle{\text{such that: }\mathrm{ED}_{\pi_i}\leq 4\;\forall i=1,\ldots, m} \\ \displaystyle{\text{and }X_{2g-2,1}=X_{2g-2,1}=0} \\ \displaystyle{\text{and }\mathrm{Aut}(\Gamma)=\{e\}}}}
\end{equation*}
where, with a slight abuse of notation, we have identified a subset of $\mathcal{U}^*_{2g-2}$ with its inverse image under $\pi_2\circ\pi_1$. We now use Proposition \ref{prp_symm1} that tells us that there exists a $C>0$ independent of $m$ such that:
\begin{equation*}
\Pro{2g-2}{\substack{\displaystyle{\exists \pi_1,\ldots,\pi_m\in\symm{2g-2}\backslash\{\mathrm{id}\}}\\ \displaystyle{\text{such that: }\mathrm{ED}_{\pi_i}\leq 4\;\forall i=1,\ldots, m} \\ \displaystyle{\text{and }X_{2g-2,1}=X_{2g-2,1}=0} \\ \displaystyle{\text{and }\mathrm{Aut}(\Gamma)=\{e\}}}} \lesssim \Pro{2g-2}{\substack{\displaystyle{X_{2g-2,1}=X_{2g-2,1}=0} \\ \displaystyle{\text{and }\sum\limits_{i=3}^{4} X_{2g-2,i} \geq Cm} }}
\end{equation*}
for $g\rightarrow\infty$. If we now apply Theorem \ref{thm_bol} we get:
\begin{equation*}
\Pro{2g-2}{\substack{\displaystyle{X_{2g-2,1}=X_{2g-2,1}=0} \\ \displaystyle{\text{and }\sum\limits_{i=3}^{4} X_{2g-2,i} < Cm} }} \sim \frac{1}{e^2} \cdot \Pro{2g-2}{\sum\limits_{i=3}^{4} X_{2g-2,i} \geq Cm}
\end{equation*}
for $g\rightarrow\infty$. Using Theorem \ref{thm_bol} again, the limit of this last probability can be made arbitrarily small by increasing $m$. This means that we have proved that for every $\delta>0$ there exists an $m_\delta$ independent of $g$ such that for all $g$ large enough there exists at least:
\begin{equation*}
(1-\delta)\cdot\aant{\mathcal{U}^*_{2g-2}}
\end{equation*}
vertices of degree $6g-6-m_\delta$ in $\MP{\Sigma_g}$. Applying Theorem \ref{thm_BC} now finishes the proof.
\end{proof}

\section{The genus of the (modular) flip graph}

For the flip graph we have the same result as for the curve graph and the pants graph:
\begin{thm}\label{thm_fg} Let $g\geq 2$ and $n>0$. We have:
$$
\gamma(\FG{\Sigma_{g,n}}) = \infty
$$
\end{thm}

\begin{proof} We will proof this by embedding a graph with vertex set $\mathbb{Z}^2\times \{0,1\}$ and edges:
\begin{equation*}
x\sim y \; \Leftrightarrow \; \abs{x-y}=1
\end{equation*}

We first consider the triangulation $T_0$ of the cylinder given in the first image of Figure \ref{pic9} below:
\begin{figure}[H]
\begin{center} 
\includegraphics[scale=0.8]{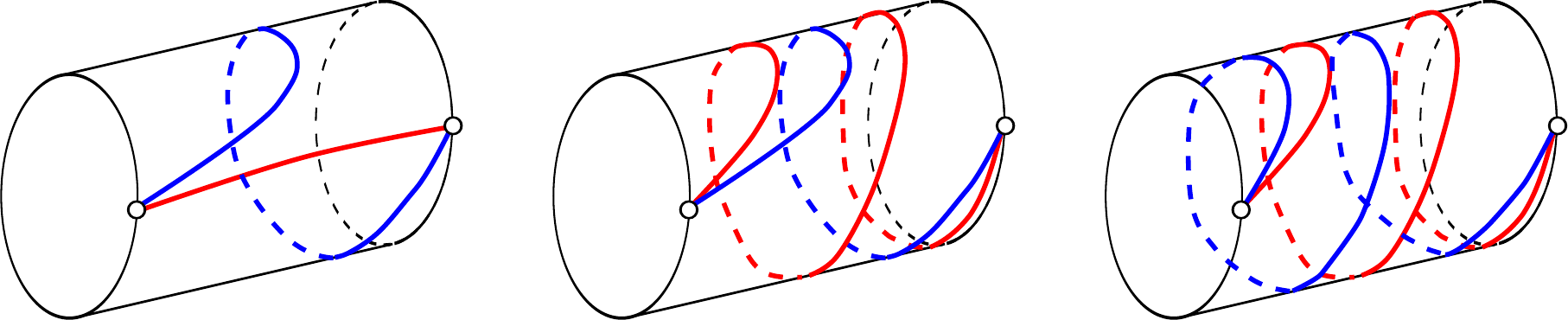} 
\caption{$T_0$, $F_rT_1$ and $T_1$.}
\label{pic9}
\end{center}
\end{figure}
Let $F_r$ denote the flip in the red curve and $F_b$ the flip in the blue curve. We now define the sequence of triangulations: $\rij{T_i}{i\in\mathbb{Z}}{}$ inductively by:
\begin{equation*}
T_{i+1} = F_bF_r T_i\text{ and } T_{i-1}= F_rF_b T_i
\end{equation*}

This means that given a triangulation of $\Sigma_{g,1}$ that contains a cylinder, we obtain a copy of $\mathbb{Z}$ in the flip graph.

We choose a triangulation $T$ of $\Sigma_{g,1}$ that contains two copies of $T_0$, with red and blue arcs $r_1,r_2$ and $b_1,b_2$ respectively and a quadrilateral. Such a triangulation $T$ exists because of our condition on $g$. The quadrilateral has flip graph $K_2$. This means that we get a set of vertices which we can label by $\rij{(T_i, T_j, k)}{(i,j,k)\in\mathbb{Z}^2\times\{0,1\}}{}$. These induce the graph above, which has infinite genus (this follows from Proposition \ref{prp_general}).
\end{proof}

For the modular flip graph we will focus on $1$-vertex triangulations. We have the following theorem:

\begin{thm}\label{thm_mf} We have:
$$
\gamma\left(\MF{\Sigma_{g,1}}\right) \;\sim_{c_1,c_2}\; \frac{1}{(4g-2)^{3/2}} \left(\frac{12g-6}{e}\right)^{2g}
$$
for $g\rightarrow\infty$, where $c_1=\frac{e}{18\sqrt{\pi}}$ and $c_2= \frac{e}{6\sqrt{\pi}}$.
\end{thm}

\begin{proof} We will use the results on random triangulated surfaces from Section \ref{sec_randsurf}. The proof is analogous to the proof for the modular pants graph. The main thing we need to control is the degree of the vertices of $\MF{\Sigma_{g,1}}$.

For the upper bound we use the fact that every triangulation $T\in\MF{\Sigma_{g,1}}$ contains $6g-3$ arcs, all of which can be flipped. This means that $\deg(T)\leq 6g-3$. Combining this with Theorem \ref{thm_cardfg1} and Proposition \ref{prp_general} gives the upper bound.

For the lower bound we again study the ratio of degree $\geq 6g-3-m$ vertices for $m\in\mathbb{N}$. In this case we do not need to worry about loops in the dual graph, because a $1$-vertex triangulation cannot contain such loops. We define the sets:
\begin{equation*}
\mathcal{O}_{g}=\left\{\text{Triangulations of $\Sigma_{g,1}$ with triangles labelled }\{1,\ldots,4g-2\} \right\}
\end{equation*}
\begin{equation*}
\mathcal{W}_{g}=\left\{\text{Isomorphism classes of triangulations of $\Sigma_{g,1}$}\right\}
\end{equation*}
So $\mathcal{W}_{4g-2}$ is the vertex set of $\MF{\Sigma_{g,1}}$ and we again have forgetful maps:
\begin{equation*}
\verz{\omega\in\Omega_{4g-2}}{S(\omega)\simeq \Sigma_{g,1}}\xrightarrow{\pi_1}\mathcal{G}_{2g-2}\xrightarrow{\pi_2}\mathcal{U}_{2g-2}
\end{equation*}
where `$\simeq$' denotes homeomorphism. We also define the set:
\begin{equation*}
\mathcal{W}_{g}^*=\verz{T\in\mathcal{W}_g}{\text{The dual graph to }T\text{ carries no non-trivial automorphism}}
\end{equation*}
Note that because in this case $2$-circuits no longer add a factor $2$ (like in the proof of Theorem \ref{thm_cardfg2}), we no longer need to add conditions on $2$ circuits in the definition of $\mathcal{W}_g^*$. Again we obtain:
\begin{equation*}
\frac{\aant{A}}{\aant{\mathcal{W}_{g}^*}}  \geq \ProC{4g-2}{\pi_1^{-1}\circ\pi_2^{-1}(A)}{S\simeq \Sigma_{g,1}} \text{ for all }A\subset\mathcal{W}^*_g
\end{equation*}
Applying the same reasoning as in the proof for $\MP{\Sigma_g}$ (but now using Proposition \ref{prp_symm2}), we obtain that there exists a constant $C>0$ independent of $m$ such that:
\begin{equation*}
\ProC{4g-2}{\deg(T)\geq 6g-3-m}{S\simeq \Sigma_{g,1}} \geq 1- \ProC{4g-2}{\sum\limits_{i=2}^{4} X_{2g-2,i} \geq Cm}{S\simeq \Sigma_{g,1}}
\end{equation*}
where we have abused notation slightly, by identifying $\pi_1^{-1}\circ\pi_2^{-1}(\mathcal{W}_g^*)$ with $\mathcal{W}_g^*$. The probability on the right hand side can be made arbitrarily small by increasing $m$. This, combined with Theorem \ref{thm_cardfg2} and Proposition \ref{prp_general} implies the lower bound.
\end{proof}

\nocite{*}
\bibliographystyle{alpha}

\end{document}